\documentclass[12pt]{article}
\usepackage[utf8]{inputenc}
\usepackage[toc,page]{appendix}
\usepackage{geometry}
\usepackage{bbm}
\usepackage{amsmath,amsthm,amssymb,amsfonts}
\usepackage{mathrsfs} 
\usepackage[colorlinks, linkcolor=blue]{hyperref}
\newtheorem{theorem}{Theorem}
\newtheorem{lemma}[theorem]{Lemma}
\newtheorem{remark}[theorem]{Remark}

\newtheorem{proposition}[theorem]{Proposition}

\numberwithin{equation}{section}
\numberwithin{theorem}{section}
\usepackage{graphicx}
\usepackage{xcolor}
\usepackage{academicons}
\usepackage{orcidlink}
\usepackage{enumitem,lipsum}
\providecommand{\keywords}[1]
{
  \small	
  \textbf{Keywords } #1
}
\providecommand{\thank}[1]
{
  \small	
  \textbf{Acknowledgment } #1
}

\title{\textbf{The defocusing Calogero–Moser derivative nonlinear Schrödinger equation with a nonvanishing condition at infinity}}
\date{}
\author{Xi Chen\thanks{Université Paris-Saclay, LMO (UMR 8628). Email: \texttt{xi.chen@universite-paris-saclay.fr}}}
\begin{document}

\maketitle
\begin{abstract}
We consider the defocusing Calogero–Moser derivative nonlinear Schrödinger equation
\begin{align*}
i \partial_{t} u+\partial_{x}^2 u-2\Pi D\left(|u|^{2}\right)u=0, \quad (t,x ) \in \mathbb{R} \times \mathbb{R}
\end{align*}
posed on $E := \left\{u \in L^{\infty}(\mathbb{R}): u' \in  L^{2}(\mathbb{R}), u'' \in L^{2}(\mathbb{R}),  |u|^{2}-1 \in L^{2}(\mathbb{R})\right\}$. We prove the global well-posedness of this equation in $E$. Moreover, we give an explicit formula for the chiral solution to this equation.
\end{abstract}
\\\\
\thank{The author is currently a PhD student at Laboratoire de Mathématique d’Orsay (UMR 8628) of Université Paris-Saclay, and he would like to thank his PhD advisor Patrick Gérard for his supervision of this paper.}
\tableofcontents{}
\section{Introduction}
We consider the following defocusing Calogero–Moser derivative nonlinear Schrödinger equation
\begin{equation}
\label{1.1}
i \partial_{t} u+\partial_{x}^2 u-2\Pi D\left(|u|^{2}\right)u=0, \quad (t,x ) \in \mathbb{R} \times \mathbb{R}.
\end{equation}
Here $\Pi$ denote the Riesz-Szeg\H{o} orthogonal projector from $L^2(\mathbb{R})$ onto 
\begin{align*}
L_{+}^2(\mathbb{R}):=\left\{f \in L^2(\mathbb{R}): \operatorname{supp}(\hat{f}) \subset[0,+\infty)\right\},
\end{align*}
and it is given by
\begin{equation}
\Pi(f)(x):=\frac{1}{2 \pi} \int_0^{\infty} e^{i \xi x} \hat{f}(\xi) d \xi, \quad \forall f \in L^2(\mathbb{R}).
\end{equation}
We remark that the defocusing Calogero–Moser derivative nonlinear Schrödinger equation was first introduced in \cite{60} under the name of intermediate nonlinear Schrödinger equation. On the other hand, the focusing Calogero–Moser derivative nonlinear Schrödinger equation was introduced in \cite{22} as a formal continuum limit of classical Calogero–Moser systems \cite{37,38}. In \cite{5}, Gérard and Lenzmann studied the focusing Calogero–Moser derivative nonlinear Schrödinger equation 
\begin{equation}
\label{1.5}
i \partial_t u+\partial_x^2 u+2\Pi D\left(|u|^2\right) u=0
\end{equation}
on the line. They derived that for any solution $u \in H_{+}^s(\mathbb{R}): = H^s(\mathbb{R}) \cap L_{+}^2(\mathbb{R})$ with $s$ sufficiently large, the equation (\ref{1.5}) enjoys a Lax pair structure and infinite conservation laws. Based on the conservation laws, they showed the global well-posedness of (\ref{1.5}) in $H_{+}^s(\mathbb{R})$ for any $s > \frac{1}{2}$ with small mass initial data ($\|u_0\|_{L^2}^2 \leq 2\pi$). Killip, Laurens and Vişan combined the explicit formulae of (\ref{1.1}) and (\ref{1.5}) with the tools and techniques of commuting
flows to deduce the global well-posedness of (\ref{1.1}) and (\ref{1.5}) (with small mass initial data) in $H_{+}^s(\mathbb{R})$ ($s \geq 0$) \cite{2}. Hogan and Kowalski showed that for arbitrarily small $\varepsilon > 0$ there exists initial data $u_0 \in H_{+}^{\infty}$ of mass $2\pi + \varepsilon$ to the equation (\ref{1.5}) such that the corresponding maximal lifespan solution $u:\left(T_{-}, T_{+}\right) \times \mathbb{R} \rightarrow \mathbb{C}$ satisfies $\lim _{t \rightarrow T_{ \pm}}\|u(t)\|_{H^s}=\infty$ for all $s>0$ \cite{9}. K. Kim, T. Kim and Kwon constructed solutions in $\mathcal{S}_{+}(\mathbb{R}):=\mathcal{S}(\mathbb{R}) \cap L_{+}^2(\mathbb{R})$ to (\ref{1.5}) with mass arbitrarily close to $2\pi$ \cite{10}. \\\\
In \cite{23} T. Kim and Kwon obtained some soliton resolution results for finite time blow up and global solutions in $H^1(\mathbb{R})$ to \eqref{1.5}. In \cite{14}, Matsuno presented multiperiodic solutions to (\ref{1.1}) and then derived multisoliton solutions. Also, Matsuno established the linear stability of solitary wave solutions \cite{15}. In \cite{16}, Matsuno constructed nonsingular $N$-phase solutions for \eqref{1.5} and explores their properties, including $N$-soliton solutions obtained via a long-wave limit. Recently, J. Chen and Pelinovsky studied the stability of nonzero constant backgrounds for both \eqref{1.1} and \eqref{1.5}, characterized traveling periodic wave solutions by using Hirota’s bilinear method, and constructed families of breathers which describe solitary waves moving across the stable background \cite{21}.
\\\\
Furthermore, Badreddine characterized the zero dispersion limit solution to (\ref{1.1}) and (\ref{1.5}) with initial data $u_0 \in L_{+}^2(\mathbb{R}) \cap L^{\infty}(\mathbb{R})$ by explicit formulae, and she identified the zero dispersion limit in terms of the branches of the multivalued solution of the inviscid Burgers–Hopf equation \cite{13}. Badreddine also showed the global well-posedness of (\ref{1.1}) and (\ref{1.5}) (with small mass initial data) on the circle in $H_{+}^s(\mathbb{T})$ for any $s \geq 0$ \cite{11}, and she characterized the traveling wave solutions to (\ref{1.1}) and (\ref{1.5}) on the circle \cite{12}. \\\\ 
The positive Fourier frequency condition $\operatorname{supp}(\hat{f}) \subset[0,+\infty)$ is interpreted as a chirality condition as in \cite{22}. As shown in the previous paragraph, most of the above studies are on chiral solutions, while in the first part of our article, we consider non-chiral solutions.\\\\
In this paper, we first study the Cauchy problem of \eqref{1.1} defined on the space
\begin{equation}
\label{1.4}
E := \left\{u \in L^{\infty}(\mathbb{R}): u' \in L^{2}\left(\mathbb{R}\right), u'' \in L^{2}\left(\mathbb{R}\right), |u|^{2}-1 \in L^{2}\left(\mathbb{R}\right)\right\}.
\end{equation}
It is easy to endow the space $E$ with a structure of complete metric space by introducing the following distance function,
\begin{equation}
d_{E}(u, \tilde{u})=\|u-\tilde{u}\|_{L^{\infty}} + \|u'-\tilde{u}'\|_{L^2}+\|u''-\tilde{u}''\|_{L^2}+\left\||u|^{2}-|\tilde{u}|^{2}\right\|_{L^{2}}.
\end{equation}
We also introduce the following Zhidkov space (see also \cite{6,8})
\begin{equation}
\label{1.606}
X^2\left(\mathbb{R}\right):=\left\{u \in L^{\infty}\left(\mathbb{R}\right): u' \in L^2(\mathbb{R}), u'' \in L^2(\mathbb{R})\right\},
\end{equation}
equipped with the norm
\begin{align*}
\|u\|_{X^2} = \|u\|_{L^{\infty}} + \|u'\|_{L^{2}} + \|u''\|_{L^{2}}.
\end{align*}
It is obvious that $E = X^2(\mathbb{R}) \cap \{|u|^2 -1 \in L^2(\mathbb{R})\}$. However, $E$ is not a vector space.\\\\
Furthermore, we introduce the following Hardy--Zhidkov space
\begin{equation}
X_{+}^2\left(\mathbb{R}\right):=\left\{u \in L_{+}^{\infty}\left(\mathbb{R}\right): u' \in L_{+}^2(\mathbb{R}), u'' \in L_{+}^2(\mathbb{R})\right\} \subset X^2(\mathbb{R}),
\end{equation}
equipped with the same norm as $\|\cdot\|_{X^2}$. In this paper, we also study the explicit formula of the solution to \eqref{1.1} in
\begin{equation}
\label{1.8}
E_{+}: = X_{+}^2\left(\mathbb{R}\right) \cap \{|u|^2 -1 \in L^2(\mathbb{R})\},
\end{equation}
where $E_{+} \subset E$.\\\\
In fact, it was Zhidkov who first introduced the following space
\begin{align*}
X^1(\mathbb{R}): =\left\{u \in L^{\infty}(\mathbb{R}): u^{\prime} \in L^2(\mathbb{R})\right\}
\end{align*}
to prove the global well-posedness results of the Gross–Pitaevskii equation \cite{6,8}
\begin{equation}
\label{1.6}
i \partial_t u+\Delta u=\left(|u|^2-1\right) u
\end{equation}
in the space 
\begin{align*}
\left\{u \in X^1(\mathbb{R}):|u|^2-1 \in L^2(\mathbb{R})\right\}.
\end{align*}
Using the strategy of Kato \cite{19}, Béthuel and Saut derived the global well-posedness of (\ref{1.6}) in $1+H^1\left(\mathbb{R}^d\right)$ for $d =2, 3$. Also, Gérard used the Strichartz estimates to obtain the global well-posedness of (\ref{1.6}) in the energy space
\begin{align*}
\left\{u \in H_{\mathrm{loc}}^1(\mathbb{R}^d): \nabla u \in L^2(\mathbb{R}^d),|u|^2-1 \in L^2(\mathbb{R}^d)\right\}
\end{align*}
for $d = 2, 3$ \cite{17}.
\subsection{Main results}
In this paper, we first study the Cauchy problem of (\ref{1.1}) defined on $E$ and we obtain the following global well-posedness result. Here, we recall the definition of $E$ from \eqref{1.4} and the definition of $X^2$ from \eqref{1.606}.
\begin{theorem}
\label{theorem 1.1}
Given $u_0 \in E$, there exists a unique solution $u \in C\left(\mathbb{R}; E\right)$ of (\ref{1.1}) with the initial data $u(0) = u_0$, and this global solution satisfies $\sup_{t \in \mathbb{R}} \|u(t)\|_{X^2} < +\infty$ and $\sup_{t \in \mathbb{R}} \||u(t)|^2-1\|_{L^2} < +\infty$. Furthermore, for every $T>0$, the flow map $u_{0} \in E \mapsto u \in C\left([-T, T]; E\right)$ is continuous.
\end{theorem}
\noindent As shown in Theorem \ref{theorem 1.1}, the global solution in $E$ has uniform bound as $t$ goes to infinity. This naturally leads to an interesting problem of the long time dynamics of global solutions in $E$. In fact, although no formal results have been established yet, $H^2(\mathbb{R})$ solutions to \eqref{1.1} are likely to satisfy scattering as the equation does not admit solitons in this setting. However, the case of \eqref{1.1} posed on $E$ might be different since the equation admits solitons in this case (see \cite{14} for details). Also, we note that the classification of traveling waves in $E$ is still unknown. Moreover, the stability of solutions in  $E$ has not been completely concluded yet (see partial results in \cite{21}).
\\\\
Now we introduce a popular tool that may help  partially resolve the problems mentioned above. In \cite{1}, Gérard established an explicit formula for the $H_{\text{real}}^2(\mathbb{R})$ solution to the Benjamin–Ono equation on the line. According to his method, Killip, Laurens and Vişan also obtained an explicit formula for the $H_{+}^s(\mathbb{R})$($s\geq 0$) solution to (\ref{1.1}) \cite{2}. In this paper, we establish an explicit formula for the solution to (\ref{1.1}) in $E_{+}$. Here, we recall the definition of $E_{+}$ from \eqref{1.8}.
\begin{theorem}
Given $u_0 \in E_{+}$, let $u \in C\left(\mathbb{R}; E_{+}\right)$ be the corresponding solution of (\ref{1.1}). Then we give the following explicit formula: For every $t \in \mathbb{R}$ and for every $z \in \mathbb{C}_{+}: = \{z \in \mathbb{C}: \operatorname{Im}(z)>0\}$, $u(t, z)$ identifies to
\begin{equation}
\label{1.7}
\begin{aligned}
& \quad \mathrm{e}^{i  t \partial_{x}^{2}} u_{0} (z) - \frac{t}{ i \pi} I_{+}\left[\left( G +2 t L_{u_{0}}-z \mathrm{Id}\right)^{-1} \left(T_{u_{0}}T_{\bar{u}_0}\mathrm{e}^{-i t \partial_{x}^{2}}  \left(\frac{\mathrm{e}^{i  t \partial_{x}^{2}} u_{0}(x)- \mathrm{e}^{i  t \partial_{x}^{2}} u_{0}(z)}{x-z}\right)\right)\right]\\ & = 
\mathrm{e}^{i  t \partial_{x}^{2}} u_{0} (z) \\ & - 2t \left[\left(Id + 2t\mathrm{e}^{i t \partial_{x}^{2}}T_{u_{0}}T_{\bar{u}_0}\mathrm{e}^{-i t \partial_{x}^{2}}(G-z I d)^{-1} \right)^{-1}\mathrm{e}^{i t \partial_{x}^{2}}T_{u_{0}}T_{\bar{u}_0}\mathrm{e}^{-i t \partial_{x}^{2}}  \left(\frac{\mathrm{e}^{i  t \partial_{x}^{2}} u_{0}(x)- \mathrm{e}^{i  t \partial_{x}^{2}} u_{0}(z)}{x-z}\right)\right] (z). 
\end{aligned}
\end{equation}
Here $u(t,z)$ and $\mathrm{e}^{i t \partial_{x}^{2}} u_{0}(z)$ are defined by the Poisson's formulation (\ref{2.6}), and the operators $G, L_{u_0}, T_{u_0}$ are defined in the subsection \ref{subsection 2.3}.
\end{theorem}
\noindent In fact, the explicit formula has already played a significant role in addressing problems related to the zero dispersion limit \cite{24, 27, 13, 25}. Thus, the first possible application of the explicit formula is to study the zero dispersion limit of solutions to \eqref{1.1} in  $E_{+}$, but a suitable entry point has not been found yet (see Section \ref{section 4} for details).\\\\ 
The explicit formula might also be useful to study the long time dynamics of solutions in $E_{+}$ since the explicit formula might provide a more direct approach to study these problems. In fact, explicit formulas have already played an essential role in the study of long time behaviors of other equations. In \cite{29}, Blackstone, Gassot, Gérard and Miller applied the explicit formula to prove that the solution to the Benjamin--Ono equation with initial datum equal to minus a soliton exhibits scattering. Also,  Gérard and Lenzmann used the explicit formula to prove soliton resolution of the half-wave maps equation for a dense subset of initial data in the energy space \cite{30}. Moreover, the explicit formula might be useful as well to study the classification of traveling waves of \eqref{1.1} in $E_{+}$.
\section{Preliminaries}
\subsection{The space $E$ and the Sobolev spaces}
Firstly, we give the following estimate.
\begin{lemma}
\label{lemma 2.1}
For any $u \in E$, we have
\begin{equation}
\|u\|_{L^{\infty}} \leq 3+C (\||u|^2-1\|_{L^2}^2 + \||u|^2-1\|_{L^2}^{2-2s}\|u'\|_{L^2}^{2s})^{1/2}
\end{equation}
for any $\frac{1}{2} < s \leq 1$.
\end{lemma}
\begin{proof}
Let $\chi \in C_{0}^{\infty}(\mathbb{C})$ such that $0 \leqslant \chi \leqslant 1, \chi(z)=1 \text { for }|z| \leqslant 2 \text {, and } \chi(z)=0 \text { for }|z| \geqslant 3 \text {. }$  Let us decompose
\begin{align*}
u=u_{1}+u_{2}, \quad u_{1}=\chi(u) u, \quad u_{2}=(1-\chi(u)) u.
\end{align*}
We have
\begin{align*}
\left\|u_{1}\right\|_{L^{\infty}} \leqslant 3.
\end{align*}
Since $|u| \geq 2$ on the support of $u_2$, we have
\begin{align*}
\left\|u_{2}\right\|_{L^{2}} \leqslant\left\||u|^{2}-1\right\|_{L^{2}}.
\end{align*}
On the other hand,
\begin{align*}
u_{2}' =(1-\chi(u)-u \partial\chi(u)) u'-u \bar{\partial} \chi(u) \bar{u}'.
\end{align*}
This implies easily, for some $A=A(\chi)$,
\begin{align*}
\left\|u_{2}'\right\|_{L^{2}} \leqslant A\|u'\|_{L^{2}}.
\end{align*}
Hence
\begin{align*}
\|u\|_{L^{\infty}} & \leq \|u_1\|_{L^{\infty}} + \|u_2\|_{L^{\infty}}\\ & \leq \|u_1\|_{L^{\infty}} + \|u_2\|_{H^s} \\ & \leq 3+C (\||u|^2-1\|_{L^2}^2 + \||u|^2-1\|_{L^2}^{2-2s}\|u'\|_{L^2}^{2s})^{1/2}
\end{align*}
for any $\frac{1}{2}<s\leq 1$.
\end{proof}
\noindent Then we introduce the following lemma which will be useful in the sequel.
\begin{lemma}
\label{lemma 2.2}
We have $E+H^2(\mathbb{R}) \subset E$, and for every $v \in E, w \in H^2(\mathbb{R})$, 
\begin{equation}
\label{2.2}
\left\||v+w|^{2}-1\right\|_{L^{2}} \leq \left\||v|^{2}-1\right\|_{L^{2}} + 2 \|v\|_{L^{\infty}} \|w\|_{L^2} + \|w\|_{L^4}^2.
\end{equation}
Moreover, for every $v \in E, \tilde{v} \in E$ and $w \in H^2, \tilde{w} \in H^2$, we have
\begin{equation}
\label{2.3}
\begin{aligned}
d_{E}(v+w, \tilde{v}+\tilde{w}) & \leq C(1+\|w\|_{L^{2}} + \|\tilde{w}\|_{L^{2}}) d_{E}(v,\tilde{v}) \\ & + C(1+\|v\|_{L^{\infty}} + \|\tilde{v}\|_{L^{\infty}} + \|w\|_{H^1}+ \|\tilde{w}\|_{H^{1}})\|w-\tilde{w}\|_{H^2}.
\end{aligned}
\end{equation}
\end{lemma}
\begin{proof}
Given $v \in E, w \in H^{2}$, we expand
\begin{align*}
|v+w|^{2}-1=|v|^{2}-1+2 \operatorname{Re}(\bar{v} w)+|w|^{2},
\end{align*}
we get the corresponding estimates (\ref{2.2}) since we have $v \in L^{\infty}(\mathbb{R})$ and $w \in L^2(\mathbb{R}) \cap L^4(\mathbb{R})$. Then we can easily deduce that $E+H^2(\mathbb{R}) \subset E$.\\\\
Also, the estimate of $\left\||v+w|^{2}-|\tilde{v}+\tilde{w}|^{2}\right\|_{L^{2}}$ can be established along the same line as (\ref{2.2}), so we can easily deduce (\ref{2.3}).
\end{proof}
\subsection{Action of the linear Schrödinger group on the space $E$}
We start with the following lemma.
\begin{lemma}
\label{lemma 2.3}
For any $f \in E$, we have $\mathrm{e}^{it\partial_x^2}f - f \in H^2(\mathbb{R})$. Moreover, we have
\begin{equation}
\label{2.4}
\|\mathrm{e}^{it\partial_x^2} f-f\|_{L^{2}} \leqslant C|t|^{1 / 2}\|f'\|_{L^{2}}
\end{equation}
and 
\begin{equation}
\label{2.5}
\|\mathrm{e}^{it\partial_x^2} f-f\|_{H^{2}} \rightarrow 0 \text{ as } t \to 0. 
\end{equation}
\end{lemma}
\begin{proof}
Let $\chi \in C_{c}^{\infty}(\mathbb{R})$ such that $\chi = 1$ near the origin, we can write
\begin{align*}
\mathrm{e}^{-\mathrm{i} t|\xi|^{2}}-1= g(t, \xi) \xi
\end{align*}
with
\begin{align*}
g(t, \xi)=-\mathrm{i} t \chi\left(t \xi^{2}\right) \xi \int_{0}^{1} \mathrm{e}^{-\mathrm{i} s t\xi^{2}} \mathrm{~d} s+\frac{1-\chi\left(t\xi^2\right)}{\xi} \left(\mathrm{e}^{-\mathrm{i} t\xi^{2}}-1\right)=\mathcal{O}\left(|t|^{1 / 2}\right).
\end{align*}
Consequently,
\begin{align*}
\mathrm{e}^{it\partial_x^2} f-f=-i g(t, D) f' \in L^{2}\left(\mathbb{R}\right)
\end{align*}
and this yields (\ref{2.4}).\\\\
Also, we can easily observe that 
\begin{align*}
\left\|\partial_x^{2}( \mathrm{e}^{it\partial_x^2} f-f)\right\|_{L^{2}} \leqslant C\left\|f''\right\|_{L^{2}} \quad \text { and } \quad\left\|\partial_x^{2}(\mathrm{e}^{it\partial_x^2} f-f)\right\|_{L^{2}} \underset{t \rightarrow 0}{\longrightarrow} 0,
\end{align*}
combined with (\ref{2.4}), we infer (\ref{2.5}).
\end{proof}
Combining the previous lemmas, it is now possible to show that the space $E$ is kept invariant by the Schrödinger group.
\begin{lemma}
\label{lemma 2.4}
For every $t \in \mathbb{R}$, $\mathrm{e}^{it\partial_x^2}(E) \subset E$ and for every $u_0 \in E$, the map
\begin{align*}
t \in \mathbb{R} \mapsto \mathrm{e}^{it\partial_x^2} u_{0} \in E
\end{align*}
is continuous. Moreover, for every $R$, for every $T>0$, there exists $C>0$ such that, for every $u_0, \tilde{u}_0 \in E$ such that $\|u_0\|_{X^2} \leqslant R$ and $\|\tilde{u}_0\|_{X^2} \leqslant R$,
\begin{align*}
\sup _{|t| \leqslant T} d_{E}\left(\mathrm{e}^{it\partial_x^2} u_{0}, \mathrm{e}^{it\partial_x^2} \tilde{u}_{0}\right) \leqslant C d_{E}\left(u_{0}, \tilde{u}_{0}\right).
\end{align*}
\end{lemma}
\begin{proof}
We write 
\begin{align*}
\mathrm{e}^{it\partial_x^2} u_{0}=u_{0}+\left(\mathrm{e}^{it\partial_x^2} u_{0}-u_{0}\right).
\end{align*}
By using Lemma \ref{lemma 2.2} and Lemma \ref{lemma 2.3},  we can easily see that $\mathrm{e}^{it\partial_x^2}(E) \subset E$, while the continuity property
\begin{align*}
\mathrm{e}^{it\partial_x^2} u_{0} \underset{t \rightarrow 0}{\longrightarrow} u_{0} \quad \text { in } E.
\end{align*}
is a consequence of (\ref{2.3}) and (\ref{2.5}). The last argument in Lemma \ref{lemma 2.4} also comes from (\ref{2.3}) and (\ref{2.5}).
\end{proof}
\begin{remark}
\label{remark 2.5}
The argument in Lemma \ref{lemma 2.4} is also valid for $u_0 \in X^2(\mathbb{R})$.
\end{remark}
\subsection{The Lax pair}
\label{subsection 2.3}
We denote by $L_{+}^p(\mathbb{R})$($1\leq p \leq \infty$) the Hardy space corresponding to $L^p(\mathbb{R})$ functions having a Fourier transform supported in the domain $\xi \geq 0$. Equivalently, the space $L_{+}^{p}(\mathbb{R})$ comprised of $f \in L^p(\mathbb{R})$ whose Poisson integral
\begin{equation}
\label{2.6}
f(z)=\int \frac{\operatorname{Im} z}{\pi|x-z|^{2}} f(x) d x 
\end{equation}
is holomorphic in the upper half-plane $\mathbb{C}_{+} : = \{z \in \mathbb{C}: \operatorname{Im}(z)>0\}$. With the Poisson's formulation (\ref{2.6}), the Hölder inequality implies
\begin{align*}
|f(z)| \lesssim(\operatorname{Im} z)^{-\frac{1}{p}}\|f\|_{L^{p}}.
\end{align*}
The Riesz-Szeg\H{o} projector $\Pi$ is the orthogonal projector from $L^2(\mathbb{R})$ onto $L_{+}^2(\mathbb{R})$. It is given by
\begin{equation}
\label{1.10}
\forall f \in L^2(\mathbb{R}), \quad \forall z \in \mathbb{C}_{+}, \quad \Pi f (z) = \frac{1}{2 \pi} \int_{0}^{\infty} \mathrm{e}^{i z \xi} \hat{f}(\xi) d \xi=   \frac{1}{2i\pi} \int_{\mathbb{R}} \frac{f(y)}{y-z} dy.
\end{equation}
The Toeplitz operator $T_b$ on $L^{2}(\mathbb{R})$ associated to a function $b \in L^{\infty}(\mathbb{R})$ is defined by
\begin{align*}
T_{b} f : =\Pi(b f), \quad f \in L^2(\mathbb{R}).
\end{align*}
For $u \in E$, the operator $L_u$ is defined by
\begin{align*}
\forall f \in \operatorname{Dom}\left(L_{u}\right)=H^{1}(\mathbb{R}), \quad  L_{u} f :=D f+uT_{\bar{u}} f \text{ with } D:=\frac{1}{i} \frac{d}{d x}.
\end{align*}
We also consider, for $u \in E$, the bounded operator defined by
\begin{align*}
B_{u}=-uT_{\partial_{x} \bar{u}}+\partial_{x} u T_{\bar{u}}+i\left(u T_{\bar{u}}\right)^{2}.
\end{align*}
Also, we recall the definition of $G$ in \cite{1},
\begin{align*}
\forall f \in \operatorname{Dom}\left(G\right): = \left\{f \in L_{+}^{2}(\mathbb{R}):  \hat{f} \in H^{1}(0, \infty)\right\},\quad  \widehat{G f}(\xi) :=i \frac{d}{d \xi}[\hat{f}(\xi)] \mathbf{1}_{\xi>0}.
\end{align*}
Here $G$ is the adjoint of the operator of multiplication by $x$ on $L_{+}^2(\mathbb{R})$, and $\left(\operatorname{Dom}\left(G\right), -iG\right)$ is maximally dissipative. Notice that $-iG$ is the infinitesimal generator of the adjoint semi–group of contractions $\left(S(\eta)^{*}\right)_{\eta \geq 0}$ with
\begin{align*}
\forall \eta \geq 0, \quad S(\eta)^{*}=\mathrm{e}^{-i \eta G}.
\end{align*}
We also notice that 
\begin{align*}
\forall f \in \operatorname{Dom}\left(G\right), \left|\hat{f}\left(0^{+}\right)\right|^2 = -4\pi\operatorname{Im}\left\langle G f \mid f\right\rangle \leq 4\pi \|Gf\|_{L^2} \|f\|_{L^2}.
\end{align*}
Therefore, we can define
\begin{align*}
\forall f \in \operatorname{Dom}\left(G\right), \quad I_{+}(f):=\hat{f}\left(0^{+}\right).
\end{align*}
In fact, as observed in \cite[Lemma 3.4]{7}, the resolvent of $G$ is given by 
\begin{equation}
\label{1.20}
\forall z \in \mathbb{C}_{+},\quad \forall f \in L_{+}^2(\mathbb{R}),\quad (G-z \operatorname{Id})^{-1} f(x)=\frac{f(x)-f(z)}{x-z},
\end{equation}
and we have
\begin{equation}
\label{1.30}
\forall z \in \mathbb{C}_{+}, \quad \forall f \in L_{+}^2(\mathbb{R}), \quad f(z)=\frac{1}{2 i \pi} I_{+}\left((G-z \operatorname{Id})^{-1} f\right).
\end{equation}
\section{Well-posedness and conservation laws}
\subsection{Local well-posedness}
In this subsection, we consider the local well-posedness of (\ref{1.1}) with initial data in $E$, where Kato’s classical iterative scheme for quasilinear evolution equations can be utilized.\\\\
We first write \eqref{1.1} as 
\begin{equation}
\label{3.1}
\partial_{t} u=i \partial_{x}^2 u-4 u \Pi(\Re(\bar{u}\partial_x u)).
\end{equation}
In view of (\ref{3.1}), we consider the following iteration scheme
\begin{equation}
\partial_{t} u^{k+1}=i \partial_{x}^2 u^{k+1}-4 u^{k} \Pi(\Re(\bar{u}^k\partial_x u^{k+1}))
\end{equation}
with initial datum $u^{k+1}(0,x) = u_0(x) \in E$. A standard energy method yields the following result.
\begin{lemma}
\label{lemma 3.2}
Let $w \in C\left([-T, T], E)\right)$ with some $T>0$ and $ w_0 \in E$. Then there exists a unique solution $w \in C\left([-T, T] ; E\right)$ such that
\begin{align*}
\partial_{t} w=i \partial_{x}^2 w - 4 u \Pi(\Re(\bar{u}\partial_x w)), \quad w(0, x)=w_{0}(x).
\end{align*}
Furthermore, if we write
\begin{align*}
w = \mathrm{e}^{it\partial_x^2} w_0 + v,
\end{align*}
then we have
\begin{equation}
\label{3.02}
\sup _{|t| \leq T}\|v(t)\|_{H^{2}} \leq C \mathrm{e}^{C \int_{-T}^{T} \|u\|_{X^2}^2 d t}\left(\|w_0'\|_{H^{1}}+ \int_{-T}^T \|u\|_{X_{+}^2}^2 \|w_0'\|_{H^1} dt\right),
\end{equation}
\end{lemma}
\begin{proof}
By Duhamel's formula, we have
\begin{align*}
v(t)= \int_{0}^t \mathrm{e}^{i(t-s)\partial_x^2} (- 4 u \Pi(\Re(\bar{u}\partial_x w)))ds,
\end{align*}
so we can deduce that
\begin{equation}
\label{3.3}
\|v(t)\|_{L^2} \leq C\int_{0}^{t} \|u\|_{L^{\infty}}^2 (\|\partial_x v\|_{L^2} + \| w_0'\|_{L^2}) ds
\end{equation}
Also, a standard energy method yields (see \cite[Lemma 5.4]{20})
\begin{align*}
\|\partial_x w(t)\|_{H^{1}} \leq \|w_0'\|_{H^{1}}+ C \int_{0}^t \|u\|_{X^2}^2\left\|\partial_x w\right\|_{H^{1}}ds,
\end{align*}
which implies
\begin{equation}
\label{3.4}
\| v(t)\|_{H^{2}} \leq C\|w_0'\|_{H^{1}} + C\int_{0}^t \left(\|u\|_{X^2}^2\left\|v\right\|_{H^{2}}+\|u\|_{X^2}^2\left\|w_0'\right\|_{H^{1}}\right)ds,
\end{equation}
Combining (\ref{3.3}) and (\ref{3.4}), by Grönwall's inequality, we can deduce that
\begin{align*}
\sup _{|t| \leq T}\|v(t)\|_{H^{2}} \leq C \mathrm{e}^{C \int_{-T}^{T} \|u\|_{X^2}^2 d t}\left(\| w_0'\|_{H^{1}}+ \int_{-T}^T \|u\|_{X^2}^2 \|w_0'\|_{H^1} dt\right),
\end{align*}
which yields (\ref{3.02}).
\end{proof}
\noindent Now we can introduce the following local well-posedness result.
\begin{proposition}
\label{theorem 3.3}
For any $R > 0$, there is some $T=T(R) > 0$ such that, for every $u_{0} \in E$ with $\|u_0\|_{X^2}^2 \leq R$, there exists a unique solution $u \in C\left([-T, T] ; E\right)$ of (\ref{1.1}) with $u(0) = u_0$. Moreover, the flow map $u_{0} \in E \mapsto u \in C\left([-T, T]; E\right)$ is continuous.
\end{proposition}
\begin{proof}
We consider the iteration scheme
\begin{equation}
\partial_{t} u^{k+1}=i \partial_{x}^2 u^{k+1}-4 u^{k} \Pi(\Re(\bar{u}^k\partial_x u^{k+1}))
\end{equation}
with initial datum $u^{k+1}(0,x) = u_0(x) \in E$. Lemma \ref{lemma 3.2} allows us to construct by induction a sequence $u^k$ with $u^{0}(t, x)=\mathrm{e}^{\mathrm{i} t \partial_x^2} u_{0}(x)$.\\\\
We write
\begin{align*}
u^{k} = \mathrm{e}^{\mathrm{i} t \partial_x^2} u_0 + v^{k}.
\end{align*}
From Lemma \ref{lemma 3.2}, we have
\begin{equation}
\label{3.6}
\sup _{|t| \leqslant T}\|v^{k+1}(t)\|_{H^2} \leq C \mathrm{e}^{C \int_{-T}^{T} \|u^k\|_{X^2}^2 d t}\left(\left\| u_{0}'\right\|_{H^{1}}+\int_{-T}^T \|u^k\|_{X^2}^2 \|u_0'\|_{H^1} dt\right).
\end{equation}
Given $u_0 \in E$ such that $\|u_0\|_{X^2} \leqslant R$, by Lemma \ref{lemma 2.3}, for $T=T(R)$ with respect to $R$,
\begin{align*}
\left\|\mathrm{e}^{i t \partial_x^2} u_0\right\|_{X^2} \leqslant 2 R.
\end{align*}
Assume that $\sup _{|t| \leqslant T} \|u^k\|_{X^2} \leq R_1$, and the size of $R_1$ will be determined later. By (\ref{3.6}), we have
\begin{align*}
\sup _{|t| \leqslant T}\|v^{k+1}\|_{H^2}  \leq C_1 e^{2 C_1 T R_{1}^{2}}(R+ 2TR_1^2 R ) := R_2.
\end{align*}
By the triangle inequality, we know that
\begin{align*}
\sup _{|t| \leqslant T} \left\|u^{k+1}\right\|_{X^2} \leq 2R + C_2 R_2.
\end{align*}
Let $R_1 = (3+C_1 C_2)R$, then we can choose $T=T(R)$ such that
\begin{align*}
2R + C_2 R_2 \leq R_1.
\end{align*}
Then by an elementary induction argument, we find that $\sup _{|t| \leqslant T} \|u^k\|_{X^2} \leq R_1$ for all $k$ and $\sup _{|t| \leqslant T}\|v^{k}\|_{H^2} \leq R_2$ for all $k$.\\\\
Next, we show that we have a contraction property of the sequence $v_k$ in $L^2(\mathbb{R})$. Observe that
\begin{align*}
\begin{array}{l}\partial_{t}\left(u^{k+1}-u^{k}\right)=i \partial_{x}^2\left(u^{k+1}-u^{k}\right)-4 u^{k}\Pi\left(\Re(\bar{u}^{k}\partial_{x}(u^{k+1}-u^{k}))\right) \\ +4 u^{k-1}\Pi\left(\Re(\bar{u}^{k-1}\partial_{x} u^{k})\right) - 4 u^{k}\Pi\left(\Re(\bar{u}^{k}\partial_{x} u^{k})\right)
\end{array}.
\end{align*}
A standard energy method yields
\begin{align*}
\sup _{|t| \leq T}\left\|u^{k+1}(t)-u^{k}(t)\right\|_{L^{2}} \leq K T \sup _{|t| \leq T}\left\|u^{k}(t)-u^{k-1}(t)\right\|_{L^{2}}.
\end{align*}
So we take $T$ small enough to ensure that $KT<1$, and we then we can deduce that $\sup _{|t| \leq T(R)}\left\|u^{k+1}(t)-u^{k}(t)\right\|_{L^{2}} = \sup _{|t| \leq T(R)}\left\|v^{k+1}(t)-v^{k}(t)\right\|_{L^{2}}$ is geometrically convergent.\\\\
Finally, the sequence $v_k(t)$ is uniformly weakly convergent in $C\left([-T, T] ; H^{2}(\mathbb{R})\right)$ and strongly convergent in $C\left([-T, T] ; L^{2}(\mathbb{R})\right)$. We note the limit of $v_k(t)$ as $v(t)$. By adapting an argument due to Bona-Smith \cite{3}(or Tao's frequency envelope method \cite{4,28}), we can deduce that the limit $v(t)$ actually belongs to $C\left([-T, T] ; H^{2}(\mathbb{R})\right)$. Thus $u(t) = \mathrm{e}^{it \partial_{x}^{2}} u_{0}+v(t) \in C([-T, T] ; E)$ solves (\ref{1.1}).\\\\
Uniqueness and continuity of the flow map follow along the same lines as the contraction property in $L^2(\mathbb{R})$. The proof of Proposition \ref{theorem 3.3} is complete.
\end{proof}
\begin{remark}
\label{remark 3.4}
In fact, by adapting the same argument used in the proof of Proposition \ref{theorem 3.3}, we can also deduce the local well-posedness of (\ref{1.1}) in $\{u\in E: u^{(3)} \in L^2(\mathbb{R})\}$.
\end{remark}
\subsection{Conservation laws}
For $u \in E$, we recall the operators $L_u$ and $B_u$ acting on $L^2 (\mathbb{R})$,
\begin{align*}
L_{u}=D+ u T_{\bar{u}} \quad \text { and } \quad B_{u}=-u T_{\partial_{x} \bar{u}}+\partial_{x} u T_{\bar{u}}+i\left(u T_{\bar{u}}\right)^{2}.
\end{align*}
We infer the following Lax equation for (\ref{1.1}).
\begin{lemma}
(Lax Equation)
If $u \in C\left([0, T] ; E\right)$ solves (\ref{1.1}), then it holds 
\begin{equation}
\label{3.7}
\frac{d}{d t} L_{u}=\left[B_{u}, L_{u}\right].
\end{equation}
\end{lemma}
\begin{proof}
We refer to \cite[Lemma 2.3]{5} for the proof.
\end{proof}
\noindent As a consequence of the Lax equation, we obtain a hierarchy of conservation laws of (\ref{1.1}), and this implies the a-priori bound on the solution.
\begin{lemma}
\label{lemma 3.5}
Let $u_0 \in E$ and $u \in C\left(I, E\right)$ be the correspoding solution of (\ref{1.1}) with $u(0,x)= u_0(x)$, where $I \subset \mathbb{R}$ denotes the maximal time interval of existence of the solution. Then the quantities
\begin{align*}
I_1(u): = \left\|D u+u \Pi\left(|u|^{2}-1\right)\right\|_{L^{2}}
\end{align*}
and
\begin{align*}
I_2 (u)  : =  \left\|D^2 u  + u T_{\bar{u}} D\left(u \right) + D(u\Pi((|u|^2 -1))) + D u + u\Pi(|u|^2\Pi(|u|^2-1)) + u\Pi(|u|^2-1) \right\|_{L^2}
\end{align*}
are conserved. \\\\
Also, we have the following a-priori bound on the solution $u$,
\begin{equation}
\label{3.080}
\sup _{t \in I}\left\||u(t)|^2-1\right\|_{L^2} \leq C\left(u_0\right),
\end{equation}
\begin{equation}
\label{3.08}
\sup _{t \in I}\|u\|_{X^2} \leq C(u_0).
\end{equation}
\end{lemma}
\begin{proof}
We observe that (\ref{1.1}) can be written as 
\begin{equation}
\label{3.8}
\partial_{t} u=-iD^2u -2i u \Pi D (|u|^2)
\end{equation}
Let
\begin{equation}
\chi_{\varepsilon}(x):=\frac{1}{1-i \varepsilon x},
\end{equation}
then we have
\begin{align*}
\partial_t (u \chi_{\varepsilon}) = \chi_{\varepsilon} \partial_t u,
\end{align*}
\begin{align*}
-iD^2(u\chi_{\varepsilon}) = i\partial_x^2 (u \chi_{\varepsilon}) = i \chi_{\varepsilon} \partial_x^2 u + 2i \partial_x u \partial_x \chi_{\varepsilon} + iu \partial_x^2\chi_{\varepsilon}
\end{align*}
and
\begin{align*}
-2i u \Pi D (|u|^2 \chi_{\varepsilon}) = -2i u \Pi(D(|u|^2) \chi_{\varepsilon})-2iu \Pi(|u|^2  D(\chi_{\varepsilon})).
\end{align*}
So we have
\begin{align*}
r_{\varepsilon} & : =\partial_t(u\chi_{\varepsilon}) - B_u(u\chi_{\varepsilon})+i(L_u)^2 (u\chi_{\varepsilon}) \\ & = \partial_t(u\chi_{\varepsilon})+iD^2(u\chi_{\varepsilon})+ 2it u  \Pi D(|u|^2 \chi_{\varepsilon}) \\ & =  -2i \partial_x u \partial_x \chi_{\varepsilon}  -iu \partial_x^2\chi_{\varepsilon} +2i u \Pi(|u|^2  D(\chi_{\varepsilon}))+2i u \Pi(D(|u|^2) \chi_{\varepsilon})-2i\chi_{\varepsilon} u \Pi(D(|u|^2)).
\end{align*}
We can easily observe that 
\begin{align*}
\|r_{\varepsilon} \|_{H^1} \underset{\varepsilon \rightarrow 0}{\longrightarrow}  0.
\end{align*}
Then we have
\begin{align*}
\frac{d}{dt}\left\langle  u \chi_{\varepsilon}, u \chi_{\varepsilon}\right\rangle &  = 2 \Re\left\langle \partial_t u \chi_{\varepsilon }, u\chi_{\varepsilon } \right\rangle \\ & = 2\Re\left\langle (B_u-iL_u^2)(u \chi_{\varepsilon}), u\chi_{\varepsilon} \right\rangle + 2\Re\left\langle r_{\varepsilon}, u \chi_{\varepsilon}\right\rangle \\ & = 2\Re\left\langle r_{\varepsilon}, u \chi_{\varepsilon}\right\rangle.
\end{align*}
Also, by the Lax pair equation (\ref{3.7}) and (\ref{3.8}), we have
\begin{align*}
& \quad \frac{d}{dt}\left\langle  L_u (u \chi_{\varepsilon}), u \chi_{\varepsilon}\right\rangle \\ & = \left\langle  L_u ( u \chi_{\varepsilon}), \partial_t (u \chi_{\varepsilon})\right\rangle + \left\langle  L_u (\partial_t (u \chi_{\varepsilon})), u \chi_{\varepsilon}\right\rangle+ \left\langle  \partial_t (L_u) (u \chi_{\varepsilon}), u \chi_{\varepsilon}\right\rangle \\& = \left\langle  L_u ( u \chi_{\varepsilon}), (B_u-iL_u^2) (u \chi_{\varepsilon})\right\rangle +  \left\langle  L_u(B_u-iL_u^2) ( u \chi_{\varepsilon}),  u \chi_{\varepsilon}\right\rangle \\ & +\left\langle  (B_u L_u-L_uB_u) ( u \chi_{\varepsilon}),  u \chi_{\varepsilon}\right\rangle + 2 \Re\left\langle L_{u}\left(u \chi_{\varepsilon}\right), r_{\varepsilon}\right\rangle\\& = 2 \Re\left\langle L_{u}\left(u \chi_{\varepsilon}\right), r_{\varepsilon}\right\rangle \\ & = 2 \Re\left\langle D\left(u \chi_{\varepsilon}\right) + u\Pi(|u|^2\chi_{\varepsilon}), r_{\varepsilon}\right\rangle \\ & = 2 \Re\left\langle D\left(u \chi_{\varepsilon}\right) + u\Pi\left((|u|^2-1)\chi_{\varepsilon}\right), r_{\varepsilon}\right\rangle + 2 \Re\left\langle u\chi_\varepsilon, r_\varepsilon\right\rangle,
\end{align*}
where
\begin{align*}
2 \Re\left\langle D\left(u \chi_{\varepsilon}\right) + u\Pi\left((|u|^2-1)\chi_{\varepsilon}\right), r_{\varepsilon}\right\rangle \underset{\varepsilon \rightarrow 0}{\longrightarrow}   0.
\end{align*}
Similarly, we have
\begin{align*}
& \quad \frac{d}{dt}\left\langle  L_u^2 (u \chi_{\varepsilon}), u \chi_{\varepsilon}\right\rangle \\ & = \left\langle  L_u^2 ( u \chi_{\varepsilon}), \partial_t (u \chi_{\varepsilon})\right\rangle + \left\langle  L_u^2 (\partial_t (u \chi_{\varepsilon})), u \chi_{\varepsilon}\right\rangle+ \left\langle  \partial_t (L_u^2) (u \chi_{\varepsilon}), u \chi_{\varepsilon}\right\rangle \\& = \left\langle  L_u^2 ( u \chi_{\varepsilon}), (B_u-iL_u^2) (u \chi_{\varepsilon})\right\rangle +  \left\langle  L_u^2(B_u-iL_u^2) ( u \chi_{\varepsilon}),  u \chi_{\varepsilon}\right\rangle \\ & +\left\langle  (B_u L_u^2-L_u^2 B_u) ( u \chi_{\varepsilon}),  u \chi_{\varepsilon}\right\rangle + 2 \Re\left\langle L_{u}^2\left(u \chi_{\varepsilon}\right), r_{\varepsilon}\right\rangle\\& = 2 \Re\left\langle L_{u}\left(u \chi_{\varepsilon}\right), L_u (r_{\varepsilon})\right\rangle \\ & = 2 \Re\left\langle D\left(u \chi_{\varepsilon}\right), L_u (r_{\varepsilon})\right\rangle+2 \Re\left\langle u \Pi\left(|u|^2 \chi_{\varepsilon}\right), L_u (r_{\varepsilon})\right\rangle\\ & = 2 \Re\left\langle D\left(u \chi_{\varepsilon}\right), L_u (r_{\varepsilon})\right\rangle+2 \Re\left\langle u \Pi\left((|u|^2-1) \chi_{\varepsilon}\right), L_u (r_{\varepsilon})\right\rangle+ 2 \Re\left\langle L_u (u  \chi_{\varepsilon}), r_{\varepsilon}\right\rangle,
\end{align*}
where 
\begin{align*}
2 \Re\left\langle D\left(u \chi_{\varepsilon}\right), L_u (r_{\varepsilon})\right\rangle+2 \Re\left\langle u \Pi\left((|u|^2-1) \chi_{\varepsilon}\right), L_u (r_{\varepsilon})\right\rangle \underset{\varepsilon \rightarrow 0}{\longrightarrow}  0 .
\end{align*}
Thus we infer that
\begin{equation}
\label{4.5}
\frac{d}{dt} \left(\left\langle L_u^2(u \chi_{\varepsilon}), u\chi_{\varepsilon} \right\rangle - 2 \left\langle L_u(u \chi_{\varepsilon}), u\chi_{\varepsilon} \right\rangle + \left\langle u \chi_{\varepsilon}, u \chi_{\varepsilon}\right\rangle\right) \underset{\varepsilon \rightarrow 0}{\longrightarrow}  0 .
\end{equation}
In fact, 
\begin{align*}
&\left\langle L_u^2(u \chi_{\varepsilon}), u\chi_{\varepsilon} \right\rangle - 2 \left\langle L_u(u \chi_{\varepsilon}), u\chi_{\varepsilon} \right\rangle + \left\langle u \chi_{\varepsilon}, u \chi_{\varepsilon}\right\rangle\\ = & \,\left\langle L_u(u \chi_{\varepsilon}), L_u (u\chi_{\varepsilon} )\right\rangle - 2 \left\langle L_u(u \chi_{\varepsilon}), u\chi_{\varepsilon} \right\rangle + \left\langle u \chi_{\varepsilon}, u \chi_{\varepsilon}\right\rangle \\ = & \,\left\langle (D+uT_{\bar{u}})(u \chi_{\varepsilon}), (D+uT_{\bar{u}})(u\chi_{\varepsilon} )\right\rangle - 2 \left\langle (D+uT_{\bar{u}})(u \chi_{\varepsilon}), u\chi_{\varepsilon} \right\rangle + \left\langle u \chi_{\varepsilon}, u \chi_{\varepsilon}\right\rangle \\ = & \,\left\langle D(u\chi_{\varepsilon})+u\Pi\left((|u|^2-1) \chi_{\varepsilon}\right)+u\chi_{\varepsilon}, D(u\chi_{\varepsilon})+u\Pi\left((|u|^2-1) \chi_{\varepsilon}\right)+u\chi_{\varepsilon}\right\rangle \\ - &\, 2 \left\langle D(u\chi_{\varepsilon})+u\Pi\left((|u|^2-1) \chi_{\varepsilon}\right)+u\chi_{\varepsilon}, u\chi_{\varepsilon} \right\rangle + \left\langle u \chi_{\varepsilon}, u \chi_{\varepsilon}\right\rangle  \\  = & \, \|D(u\chi_{\varepsilon}) + u\Pi\left((|u|^2-1)\chi_{\varepsilon}\right)\|_{L^2}^2.
\end{align*}
Then by (\ref{4.5}), we can deduce that
\begin{align*}
\|D(u(t)\chi_{\varepsilon}) + u(t)\Pi\left((|u(t)|^2-1)\chi_{\varepsilon}\right)\|_{L^2}^2 - \|D(u_0\chi_{\varepsilon}) + u_0\Pi\left((|u_0|^2-1)\chi_{\varepsilon}\right)\|_{L^2}^2 \underset{\varepsilon \rightarrow 0}{\longrightarrow}  0,
\end{align*}
and thus
\begin{align*}
I_1(u): = \|Du(t) + u(t)\Pi\left(|u(t)|^2-1\right)\|_{L^2} = \|Du_0 + u_0\Pi\left(|u_0|^2-1\right)\|_{L^2} = I_1(u_0),
\end{align*}
which yields the first conservation law in Lemma \ref{lemma 3.5}.\\\\
Following the modified gauge transform
\begin{align*}
v := u(x) \mathrm{e}^{\frac{i}{2} \int_{0}^x (|u(y)|^2 -1)d y},
\end{align*}
we know that
\begin{align*}
I_1(u)  = \|Du + u\Pi\left(|u|^2-1\right)\|_{L^2} = \|\partial_x v - \frac{1}{2} v \mathrm{H}\left(|v|^2-1\right)\|_{L^2}.
\end{align*}
Here we use $\Pi = \frac{1}{2}(\operatorname{Id} + i\mathrm{H})$, where $\mathrm{H}$ denotes the Hilbert transform.\\\\
From \cite[Appendix C]{5} we infer
\begin{align*}
& \quad \Re\left\langle\partial_x v, \mathrm{H}\left(|v|^2-1\right) v\right\rangle \\ & =\left\langle\Re\left[\bar{v} \partial_x v\right], \mathrm{H}\left(|v|^2-1\right)\right\rangle \\ & = \frac{1}{2}\left\langle\partial_x\left(|v|^2-1\right), \mathrm{H}\left(|v|^2-1\right)\right\rangle \\ & = - \frac{1}{2}\left\langle | v|^2 -1, \mathrm{H} \partial_x \left(|v|^2-1\right)\right\rangle \\ & =-\frac{1}{2}\left\langle | v|^2 -1, |D|\left(|v|^2-1\right)\right\rangle
\end{align*}
and 
\begin{align*}
& \quad \left\langle\mathrm{H}\left(|v|^2-1\right) v, \mathrm{H}\left(|v|^2-1\right) v\right\rangle \\ & = \left\langle |v|^2, \left(\mathrm{H}\left(|v|^2-1\right)\right)^2\right\rangle \\ & = \left\langle |v|^2-1, \left(\mathrm{H}\left(|v|^2-1\right)\right)^2\right\rangle + \left\langle 1, \left(\mathrm{H}\left(|v|^2-1\right)\right)^2\right\rangle  \\ & = \left\langle |v|^2-1, \left(\mathrm{H}\left(|v|^2-1\right)\right)^2\right\rangle + \left\langle 1, \left(\mathrm{H}\left(|v|^2-1\right)\right)^2\right\rangle \\ & = \frac{1}{3} \int_{\mathbb{R}} (|v|^2 -1)^3 dx + \||v|^2-1\|_{L^2}^2,
\end{align*}
so we have
\begin{align*}
&\quad \|\partial_x v - \frac{1}{2} v \mathrm{H}\left(|v|^2-1\right)\|_{L^2}^2 \\ & = \|\partial_x v\|_{L^2}^2 + \frac{1}{2}\| |v|^2-1\|_{\dot{H}^{\frac{1}{2}
}}^2 +\frac{1}{12} \int_{\mathbb{R}} (|v|^2-1)^3 dx + \frac{1}{4}\||v|^2-1\|_{L^2}^2 \\ & \geq \frac{1}{12} \int_{|v|\geq 1} (|v|^2-1)^3 dx + \frac{1}{12} \int_{|v|<1} (|v|^2-1)^3 dx + \frac{1}{4}\||v|^2-1\|_{L^2}^2 \\ & \geq \frac{1}{4} \int_{|v|\geq 1} (|v|^2-1)^2 dx + \frac{1}{6} \int_{|v|<1} (|v|^2-1)^2 dx\\ & \geq \frac{1}{6} 
\||v|^2-1\|_{L^2}^2,
\end{align*}
Thus
\begin{align*}
I_1(u_0)^2 = I_1(u)^2  = \|\partial_x v - \frac{1}{2} v \mathrm{H}\left(|v|^2-1\right)\|_{L^2}^2 \geq \frac{1}{6}\||v|^2-1\|_{L^2}^2 = \frac{1}{6}\||u|^2-1\|_{L^2}^2.
\end{align*}
This implies $\sup _{t \in I}\left\||u(t)|^2 -1 \right\|_{L^{2}} \leq C\left(u_{0}\right)$, which is
\eqref{3.080}.\\\\
According to Lemma \ref{lemma 2.1}, for every $\frac{1}{2} < s <1$, we have
\begin{align*}
& \quad \|Du + u\Pi\left(|u|^2-1\right)\|_{L^2}\\  & \geq \left\|\partial_{x} u\right\|_{L^{2}}^{2} - 2\left\|\partial_{x}u\right\|\|u\|_{L^{\infty}} \|\Pi\left(|u|^2-1\right)\|_{L^2} \\ & \geq \left\|\partial_{x} u\right\|_{L^{2}}^{2} -6 \left\|\partial_{x} u\right\|_{L^{2}}\|\Pi\left(|u|^2-1\right)\|_{L^2}\\& -C\left\|\partial_{x} u\right\|_{L^{2}}\left(\left\||u|^{2}-1\right\|_{L^{2}}^{2}+\left\||u|^{2}-1\right\|_{L^{2}}^{2-2 s}\left\|\partial_x u\right\|_{L^{2}}^{2 s}\right)^{1 / 2} \|\Pi\left(|u|^2-1\right)\|_{L^2}.
\end{align*}
From the a-priori bound on $\left\||u|^{2}-1\right\|_{L^{2}}$, we readily infer that  $\sup _{t \in I}\left\|\partial_{x} u(t)\right\|_{L^{2}} \leq C\left(u_{0}\right)$. Again by Lemma \ref{lemma 2.1}, we can deduce that  $\sup _{t \in I}\left\|u(t)\right\|_{L^{\infty}} \leq C\left(u_{0}\right)$.\\\\
Now we assume moreover that $u_0^{(3)} \in L^2 (\mathbb{R})$. By Remark \ref{remark 3.4}, we know that $\partial_x^3 u(t) \in L^2 (\mathbb{R})$, and then we can easily observe that 
\begin{align*}
\left\|r_{\varepsilon}\right\|_{H^{2}} \underset{\varepsilon \rightarrow 0}{\longrightarrow} 0
\end{align*}
in this case. \\\\
Now we have
\begin{align*}
& \quad \frac{d}{dt}\left\langle  L_u^4 (u \chi_{\varepsilon}), u \chi_{\varepsilon}\right\rangle \\ & = \left\langle  L_u^4 ( u \chi_{\varepsilon}), (\partial_t u \chi_{\varepsilon})\right\rangle + \left\langle  L_u^4 (\partial_t (u \chi_{\varepsilon})), u \chi_{\varepsilon}\right\rangle+ \left\langle  \partial_t (L_u^4) (u \chi_{\varepsilon}), u \chi_{\varepsilon}\right\rangle \\& = \left\langle  L_u^4 ( u \chi_{\varepsilon}), (B_u-iL_u^2) (u \chi_{\varepsilon})\right\rangle +  \left\langle  L_u^4(B_u-iL_u^2) ( u \chi_{\varepsilon}),  u \chi_{\varepsilon}\right\rangle \\ & +\left\langle  (B_u L_u^4-L_u^4 B_u) ( u \chi_{\varepsilon}),  u \chi_{\varepsilon}\right\rangle + 2 \Re\left\langle L_{u}^4\left(u \chi_{\varepsilon}\right), r_{\varepsilon}\right\rangle\\& = 2 \Re\left\langle L_{u}^2\left(u \chi_{\varepsilon}\right), L_u^2 (r_{\varepsilon})\right\rangle \\ & = 2 \Re\left\langle D^2\left(u \chi_{\varepsilon}\right), L_u (r_{\varepsilon})\right\rangle+2 \Re\left\langle u T_{\bar{u}} D\left( u\chi_{\varepsilon}\right), L_u^2 (r_{\varepsilon})\right\rangle + 2 \Re\left\langle D \left(uT_{\bar{u}}\left( u \chi_{\varepsilon}\right)\right), L_u^2  (r_{\varepsilon})\right\rangle\\ & + 2 \Re \left\langle u\Pi\left(|u|^2\Pi\left(|u|^2 -1)\chi_{\varepsilon}\right)\right), L_u^2  (r_{\varepsilon}) \right\rangle + 2 \Re \left\langle u\Pi((|u|^2-1) \chi_{\varepsilon}), L_u^2  (r_{\varepsilon})\right\rangle \\ & + 2 \Re\left\langle L_{u}\left(u \chi_{\varepsilon}\right), L_{u}\left(r_{\varepsilon}\right)\right\rangle,
\end{align*}
where
\begin{align*}
& 2 \Re\left\langle D^2\left(u \chi_{\varepsilon}\right), L_u (r_{\varepsilon})\right\rangle+2 \Re\left\langle u T_{\bar{u}} D\left( \chi_{\varepsilon}\right), L_u^2 (r_{\varepsilon})\right\rangle + 2 \Re\left\langle D \left(uT_{\bar{u}}\left( u \chi_{\varepsilon}\right)\right), L_u^2  (r_{\varepsilon})\right\rangle\\ & + 2 \Re \left\langle u\Pi\left(|u|^2\Pi\left(|u|^2 -1)\chi_{\varepsilon}\right)\right), L_u^2  (r_{\varepsilon}) \right\rangle + 2 \Re \left\langle u\Pi((|u|^2-1) \chi_{\varepsilon}), L_u^2  (r_{\varepsilon})\right\rangle \underset{\varepsilon \rightarrow 0}{\longrightarrow} 0. 
\end{align*}
Then we infer that
\begin{align*}
\frac{d}{d t}\left(\left\langle L_{u}^{4}\left(u \chi_{\varepsilon}\right), u \chi_{\varepsilon}\right\rangle- 2\left\langle L_{u}^{2}\left(u \chi_{\varepsilon}\right), u \chi_{\varepsilon}\right\rangle  + \left\langle u \chi_{\varepsilon}, u \chi_{\varepsilon}\right\rangle\right) \underset{\varepsilon \rightarrow 0}{\longrightarrow} 0.
\end{align*}
We also observe that
\begin{align*}
& \quad \left\langle L_{u}^{4}\left(u \chi_{\varepsilon}\right), u \chi_{\varepsilon}\right\rangle- 2\left\langle L_{u}^{2}\left(u \chi_{\varepsilon}\right), u \chi_{\varepsilon}\right\rangle  + \left\langle u \chi_{\varepsilon}, u \chi_{\varepsilon}\right\rangle  \\ & = \|D^2(u\chi_{\varepsilon} ) + u T_{\bar{u}} D\left(u \chi_{\varepsilon}\right) + D(u\Pi((|u|^2 -1)\chi_{\varepsilon})) + D(u\chi_{\varepsilon}) + u\Pi(|u|^2\Pi((|u|^2-1)\chi_{\varepsilon})) \\ & + u\Pi((|u|^2-1)\chi_{\varepsilon}) \|_{L^2}^2 \\ & : = I_{2,\varepsilon}(u)^2.
\end{align*}
Then we can deduce that
\begin{align*}
I_{2,\varepsilon}(u(t))^2 - I_{2,\varepsilon}(u_0)^2 \underset{\varepsilon \rightarrow 0}{\longrightarrow} 0,
\end{align*}
which implies that
\begin{equation}
\label{4.6}
I_2(u(t)) = I_2(u_0).
\end{equation}
By the density and the continuity of the flow map, we can deduce (\ref{4.6}) for $u_0 \in E$, which yields the second conservation law in Lemma \ref{lemma 3.5}.\\\\
Then from the a-priori bounds on $\|u(t)\|_{L^{\infty}}$, $\|\partial_x u(t)\|_{L^2}$ and $\||u|^2 -1 \|_{L^2}$, we can infer that $\sup _{t \in I}\left\|\partial_{x}^2 u(t)\right\|_{L^{2}} \leq C\left(u_{0}\right)$, thus we have (\ref{3.08}).\\\\
The proof is complete.
\end{proof}
\subsection{Global well-posedness}
As an application of Lemma \ref{lemma 3.5}, we deduce the following global well-posedness result. 
\begin{theorem}
Given $u_0 \in E$, there exists a unique solution $u \in C\left(\mathbb{R}; E\right)$ of (\ref{1.1}) with initial data $u(0) = u_0$, and this global solution satisfies $\sup_{t \in \mathbb{R}} \|u(t)\|_{X^2} < +\infty$ and $\sup_{t \in \mathbb{R}} \||u(t)|^2-1\|_{L^2} < +\infty$. Furthermore, for every $T>0$, the flow map $u_{0} \in E \mapsto u \in C\left([-T, T]; E\right)$ is continuous.
\end{theorem}
\begin{proof}
From Lemma \ref{lemma 3.5}, we know that $\|u(t)\|_{X^2}$ and $\||u|^2-1\|_{L^2}$ is uniformly bounded in $t$, then by the classical bootstrap argument, we infer that the solution exists globally (with uniform bound) in time, which yields the global well-posedness of (\ref{1.1}) in $E$. 
\end{proof}
\section{Explicit formula for chiral solutions}
\label{section 4}
In \cite{2}, Killip, Laurens and Vişan established an explicit formula of the $H_{+}^s(\mathbb{R})$($s\geq 0$) solution to (\ref{1.1}). In this section, we establish an explicit formula of the solution to (\ref{1.1}) in $E_{+}$.
\begin{theorem}
\label{theorem 4.1}
Given $u_0 \in E_{+}$, let $u \in C\left(\mathbb{R}; E_{+}\right)$ be the corresponding solution of (\ref{1.1}). Then we give the following explicit formula: For every $t \in \mathbb{R}$ and for every $z \in \mathbb{C}_{+}: = \{z \in \mathbb{C}: \operatorname{Im}(z)>0\}$, $u(t, z)$ identifies to
\begin{equation}
\label{4.1}
\begin{aligned}
& \quad \mathrm{e}^{i  t \partial_{x}^{2}} u_{0} (z) - \frac{t}{ i \pi} I_{+}\left[\left( G +2 t L_{u_{0}}-z \mathrm{Id}\right)^{-1} \left(T_{u_{0}}T_{\bar{u}_0}\mathrm{e}^{-i t \partial_{x}^{2}}  \left(\frac{\mathrm{e}^{i  t \partial_{x}^{2}} u_{0}(x)- \mathrm{e}^{i  t \partial_{x}^{2}} u_{0}(z)}{x-z}\right)\right)\right]\\ & = 
\mathrm{e}^{i  t \partial_{x}^{2}} u_{0} (z) \\ & - 2t \left[\left(Id + 2t\mathrm{e}^{i t \partial_{x}^{2}}T_{u_{0}}T_{\bar{u}_0}\mathrm{e}^{-i t \partial_{x}^{2}}(G-z I d)^{-1} \right)^{-1}\mathrm{e}^{i t \partial_{x}^{2}}T_{u_{0}}T_{\bar{u}_0}\mathrm{e}^{-i t \partial_{x}^{2}}  \left(\frac{\mathrm{e}^{i  t \partial_{x}^{2}} u_{0}(x)- \mathrm{e}^{i  t \partial_{x}^{2}} u_{0}(z)}{x-z}\right)\right] (z). 
\end{aligned}
\end{equation}
Here $u(t,z)$ and $\mathrm{e}^{i t \partial_{x}^{2}} u_{0}(z)$ are defined by the Poisson's formulation (\ref{2.6}).
\end{theorem}
\begin{proof}
We satrt with the derivation method as in \cite[Section 3]{1}. For every $f \in L_{+}^2(\mathbb{R})$, from (\ref{1.10}) we know that
\begin{align*}
\forall z \in \mathbb{C}_{+}, \quad f(z)=\frac{1}{2 \pi} \int_{0}^{\infty} \mathrm{e}^{i z \xi} \hat{f}(\xi) d \xi.
\end{align*}
While, in view of the Plancherel theorem, we have, in $L^{2}(0,+\infty)$,
\begin{align*}
\hat{f}(\xi)=\lim _{\delta \rightarrow 0} \int_{\mathbb{R}} \mathrm{e}^{-i x \xi} \frac{f(x)}{1+i \delta x} d x=\lim _{\delta \rightarrow 0}\left\langle S(\xi)^{*} f, \chi_{\delta}\right\rangle,
\end{align*}
where 
\begin{align*}
\chi_{\delta}(x):=\frac{1}{1-i \delta x}.
\end{align*}
Plugging the second formula into the first one, we infer
\begin{align*}
f(z) & =\lim _{\delta \rightarrow 0} \frac{1}{2 \pi} \int_{0}^{\infty} \mathrm{e}^{i z \xi}\left\langle S(\xi)^{*} f, \chi_{\delta}\right\rangle d \xi \\ & =\lim _{\delta \rightarrow 0} \frac{1}{2 \pi} \int_{0}^{\infty} \mathrm{e}^{i z \xi}\left\langle\mathrm{e}^{-i \xi G} f, \chi_{\delta}\right\rangle d \xi \\ & =\lim _{\delta \rightarrow 0} \frac{1}{2 i \pi}\left\langle(G-z \mathrm{Id})^{-1} f, \chi_{\delta}\right\rangle.
\end{align*}
We then use the family $U(t)$ of unitary operators defined by the linear initial value problem in $\mathscr{L}\left(L_{+}^{2}(\mathbb{R})\right)$,
\begin{align*}
U^{\prime}(t)=B_{u(t)} U(t), U(0)=\mathrm{Id}.
\end{align*}
For every $z \in \mathbb{C}_{+}$, we have
\begin{align*}
(u\chi_{\varepsilon})(t, z) & =\lim _{\delta \rightarrow 0} \frac{1}{2 i \pi}\left\langle U(t)^{*}(G-z \mathrm{Id})^{-1} (u\chi_{\varepsilon})(t) , U(t)^{*} \chi_{\delta}\right\rangle \\ & =\lim _{\delta \rightarrow 0} \frac{1}{2 i \pi}\left\langle\left(U(t)^{*} G U(t)-z \mathrm{Id}\right)^{-1} U(t)^{*} (u\chi_{\varepsilon})(t), U(t)^{*} \chi_{\delta}\right\rangle.
\end{align*}
It is not difficult to see that
\begin{equation}
\left[G, B_{u}\right]=2 L_{u}+i\left[G,\left(L_{u}\right)^{2}\right],
\end{equation}
then we calculate
\begin{align*}
\frac{d}{d t} U(t)^{*} G U(t) & =U(t)^{*}\left[G, B_{u(t)}\right] U(t) \\ & =U(t)^{*}\left(2 L_{u(t)}+i\left[G, L_{u(t)}^{2}\right]\right) U(t) \\ & = 2 L_{u_{0}}+i\left[U(t)^{*} G U(t), L_{u_{0}}^{2}\right].
\end{align*}
Integrating this ODE, we get
\begin{align*}
U(t)^{*} G U(t)=2 t L_{u_{0}}+\mathrm{e}^{-i t L_{u_{0}}^{2}} G \mathrm{e}^{i t L_{u_{0}}^{2}}.
\end{align*}
Let us determine the other terms in the inner product. We have 
\begin{align*}
\frac{d}{d t} U(t)^{*} (u\chi_{\varepsilon}) & =U(t)^{*}\left(\partial_{t} (u\chi_{\varepsilon})-B_{u(t)}  (u\chi_{\varepsilon})\right) \\ & =-i U(t)^{*} L_{u(t)}^{2} (u\chi_{\varepsilon}) + U(t)^{*} r_{\varepsilon}\\ &=-i L_{u_{0}}^{2} U(t)^{*} ( u\chi_{\varepsilon}) + U(t)^{*}r_{\varepsilon},
\end{align*}
from which we infer
\begin{align*}
R_{\varepsilon} : = U(t)^{*} (u\chi_{\varepsilon}) - \mathrm{e}^{-i t L_{u_{0}}^{2}}  (u_{0} \chi_{\varepsilon}) \underset{\varepsilon \rightarrow 0}{\longrightarrow} 0 
\end{align*}
in $L_{+}^2(\mathbb{R})$.\\\\
Moreover, we can easily observe that
\begin{align*}
B_{u}\chi_{\delta} - i L_u^2 \chi_{\delta} \underset{\delta \rightarrow 0}{\longrightarrow} 0 
\end{align*}
in $L_{+}^2(\mathbb{R})$.\\\\
So we have
\begin{align*}
\frac{d}{d t} U(t)^{*} \chi_{\delta} & = -U(t)^{*} B_{u}\chi_{\delta} \\ & = -i\lim_{\delta \rightarrow 0}  U(t)^{*} (L_u^2 \chi_{\delta}) \\ & = -i\lim_{\delta \rightarrow 0} L_{u_{0}}^2 U(t)^{*}\chi_{\delta}
\end{align*}
By integrating in time, we infer
\begin{align*}
U(t)^{*} \chi_{\delta}-\mathrm{e}^{-i t L_{u_{0}}^{2}} \chi_{\delta} \underset{\delta \rightarrow 0}{\longrightarrow} 0 
\end{align*}
in $L_{+}^2(\mathbb{R})$. Plugging these informations into the formula which gives $u\chi_{\varepsilon}$, we infer 
\begin{align*}
&\quad (u\chi_{\varepsilon}) (t,z)\\ & =\lim _{\delta \rightarrow 0} \frac{1}{2 i \pi}\left\langle\left(U(t)^{*} G U(t)-z \mathrm{Id}\right)^{-1} U(t)^{*} (u\chi_{\varepsilon})(t), U(t)^{*} \chi_{\delta}\right\rangle \\ & = \lim _{\delta \rightarrow 0} \frac{1}{2 i \pi}\left\langle\left(\mathrm{e}^{-i t L_{u_{0}}^{2}} G \mathrm{e}^{i t L_{u_{0}}^{2}}+2 t L_{u_{0}}-z \mathrm{Id}\right)^{-1} \left(\mathrm{e}^{-i t L_{u_{0}}^{2}} (u_{0}\chi_{\varepsilon}) + R_{\varepsilon}\right), \mathrm{e}^{-i t L_{u_{0}}^{2}} \chi_{\delta}\right\rangle \\ & = \lim _{\delta \rightarrow 0} \frac{1}{2 i \pi}\left\langle\left( G +2 t L_{u_{0}}-z \mathrm{Id}\right)^{-1} \left(u_{0}\chi_{\varepsilon} + \mathrm{e}^{i t L_{u_{0}}^{2}} R_{\varepsilon}\right), \chi_{\delta}\right\rangle \\ & = \frac{1}{2 i \pi} I_{+}\left[\left( G +2 t L_{u_{0}}-z \mathrm{Id}\right)^{-1} \left(u_{0}\chi_{\varepsilon} + \mathrm{e}^{i t L_{u_{0}}^{2}} R_{\varepsilon}\right)\right]
\end{align*}
We can easily see that
\begin{align*}
(u\chi_{\varepsilon}) (t,z) \underset{\varepsilon \rightarrow 0}{\longrightarrow} u (t,z), \quad \forall z \in \mathbb{C}_{+}
\end{align*}
and
\begin{align*}
\frac{1}{2 i \pi} I_{+}\left[\left( G +2 t L_{u_{0}}-z \mathrm{Id}\right)^{-1} R_{\varepsilon}\right] \underset{\varepsilon \rightarrow 0}{\longrightarrow} 0,  \quad \forall z \in \mathbb{C}_{+},
\end{align*}
so we only need to give the limit of $\frac{1}{2 i \pi} I_{+}\left[\left( G +2 t L_{u_{0}}-z \mathrm{Id}\right)^{-1} \left(u_{0}\chi_{\varepsilon}\right)\right]$ as $\varepsilon$ tends to 0.\\\\
In fact, we have
\begin{align*}
 G +2 t L_{u_{0}}-z & = G+2tD+2tT_{u_0} T_{\bar{u}_0} -z Id \\ & = \mathrm{e}^{-it\partial_x^2} G \mathrm{e}^{it\partial_x^2} + 2t T_{u_0} T_{\bar{u}_0} -z Id \\ & = \mathrm{e}^{-it\partial_x^2} \left(G+ 2t \mathrm{e}^{it\partial_x^2} T_{u_0} T_{\bar{u}_0}\mathrm{e}^{-it\partial_x^2} -z Id\right)\mathrm{e}^{it\partial_x^2}.
\end{align*}
Combining the above formula and (\ref{1.30}), we infer
\begin{align*}
& \quad \frac{1}{2 i \pi} I_{+}\left[\left( G +2 t L_{u_{0}}-z \mathrm{Id}\right)^{-1} \left(u_{0}\chi_{\varepsilon}\right)\right] \\ & = \frac{1}{2 i \pi} I_{+}\left[\left(G+2 t \mathrm{e}^{i  t \partial_{x}^{2}} T_{u_{0}} T_{\bar{u}_0}\mathrm{e}^{-i t \partial_{x}^{2}}-z \mathrm{Id}\right)^{-1} \mathrm{e}^{i  t \partial_{x}^{2}} (u_{0}\chi_{\varepsilon})\right] \\ & = \left[ \left(Id+2 t \mathrm{e}^{i  t \partial_{x}^{2}} T_{u_{0}} T_{\bar{u}_0}\mathrm{e}^{-i t \partial_{x}^{2}}(G-zId)^{-1}\right)^{-1} \mathrm{e}^{i  t \partial_{x}^{2}} (u_{0}\chi_{\varepsilon})\right](z).
\end{align*}
From the identity
\begin{align*}
(Id + A)^{-1} = Id - (Id + A)^{-1} A,
\end{align*}
we know that
\begin{align*}
& \quad \left(Id+2 t \mathrm{e}^{i  t \partial_{x}^{2}} T_{u_{0}} T_{\bar{u}_0}\mathrm{e}^{-i t \partial_{x}^{2}}(G-zId)^{-1}\right)^{-1} \\ & = Id - 2t \left(Id+2 t \mathrm{e}^{i  t \partial_{x}^{2}} T_{u_{0}} T_{\bar{u}_0}\mathrm{e}^{-i t \partial_{x}^{2}}(G-zId)^{-1}\right)^{-1} \mathrm{e}^{i  t \partial_{x}^{2}} T_{u_{0}} T_{\bar{u}_0}\mathrm{e}^{-i t \partial_{x}^{2}} (G-zId)^{-1}.
\end{align*}
Then we infer 
\begin{align*}
 & \quad \left(Id+2 t \mathrm{e}^{i  t \partial_{x}^{2}} T_{u_{0}} T_{\bar{u}_0}\mathrm{e}^{-i t \partial_{x}^{2}}(G-zId)^{-1}\right)^{-1} \mathrm{e}^{i  t \partial_{x}^{2}} (u_{0}\chi_{\varepsilon}) \\ & = \mathrm{e}^{i  t \partial_{x}^{2}} (u_{0}\chi_{\varepsilon})\\ & - 2t \left(Id + 2t\mathrm{e}^{i t \partial_{x}^{2}}T_{u_{0}}T_{\bar{u}_0}\mathrm{e}^{-i t \partial_{x}^{2}}(G-z I d)^{-1} \right)^{-1}\mathrm{e}^{i t \partial_{x}^{2}}T_{u_{0}}T_{\bar{u}_0}\mathrm{e}^{-i t \partial_{x}^{2}} (G-zId)^{-1} \mathrm{e}^{i  t \partial_{x}^{2}} (u_{0}\chi_{\varepsilon}) \\ & = \mathrm{e}^{i  t \partial_{x}^{2}} (u_{0}\chi_{\varepsilon})\\ & - 2t \left(Id + 2t\mathrm{e}^{i t \partial_{x}^{2}}T_{u_{0}}T_{\bar{u}_0}\mathrm{e}^{-i t \partial_{x}^{2}}(G-z I d)^{-1} \right)^{-1}\mathrm{e}^{i t \partial_{x}^{2}}T_{u_{0}}T_{\bar{u}_0}\mathrm{e}^{-i t \partial_{x}^{2}}  \left(\frac{\mathrm{e}^{i  t \partial_{x}^{2}} (u_{0}\chi_{\varepsilon})(x)- \mathrm{e}^{i  t \partial_{x}^{2}} (u_{0}\chi_{\varepsilon})(z)}{x-z}\right).
\end{align*}
The second equality above comes from (\ref{1.20}).\\\\
Let $\varepsilon$ tend to 0, we know that
\begin{align*}
u_0 \chi_{\varepsilon} \underset{\varepsilon \rightarrow 0}{\longrightarrow} u_0 \text{ pointwisely }
\end{align*}
and
\begin{align*}
(u_0 \chi_{\varepsilon})' \underset{\varepsilon \rightarrow 0}{\longrightarrow} u_0' \text{ in } H_{+}^1(\mathbb{R}).
\end{align*}
By Lemma \ref{lemma 2.3}, we infer
\begin{align*}
\mathrm{e}^{i  t \partial_{x}^{2}} (u_{0}\chi_{\varepsilon})  = u_{0}\chi_{\varepsilon} +\left(\mathrm{e}^{i  t \partial_{x}^{2}} (u_{0}\chi_{\varepsilon})-u_0\chi_{\varepsilon}\right)  \underset{\varepsilon \rightarrow 0}{\longrightarrow} u_{0}+\left(\mathrm{e}^{i  t \partial_{x}^{2}} u_{0}-u_0\right)= \mathrm{e}^{i  t \partial_{x}^{2}} u_{0} \text{ pointwisely. }
\end{align*}
Then from the Poisson's formulation (\ref{2.6}) and the dominated convergence theorem, we can deduce that
\begin{align*}
\mathrm{e}^{i  t \partial_{x}^{2}} (u_{0}\chi_{\varepsilon})(z) \underset{\varepsilon \rightarrow 0}{\longrightarrow} \mathrm{e}^{i  t \partial_{x}^{2}} u_{0} (z), \quad \forall z \in \mathbb{C}_{+}.
\end{align*}
Again by the dominated convergence theorem, we can deduce that, for every $z \in \mathbb{C}_{+}$,
\begin{align*}
\frac{\mathrm{e}^{i  t \partial_{x}^{2}} (u_{0}\chi_{\varepsilon})(x)- \mathrm{e}^{i  t \partial_{x}^{2}} (u_{0}\chi_{\varepsilon})(z)}{x-z} \underset{\varepsilon \rightarrow 0}{\longrightarrow} \frac{\mathrm{e}^{i  t \partial_{x}^{2}} u_{0} (x)- \mathrm{e}^{i  t \partial_{x}^{2}} u_{0}(z)}{x-z} \text{ in } L_{+}^2(\mathbb{R}).
\end{align*}
Thus we have
\begin{align*}
& \quad \lim_{\varepsilon \rightarrow 0} \frac{1}{2 i \pi} I_{+}\left[\left( G +2 t L_{u_{0}}-z \mathrm{Id}\right)^{-1} \left(u_{0}\chi_{\varepsilon}\right)\right] \\ & = \mathrm{e}^{i  t \partial_{x}^{2}} u_{0} (z) \\ & - 2t \left[\left(Id + 2t\mathrm{e}^{i t \partial_{x}^{2}}T_{u_{0}}T_{\bar{u}_0}\mathrm{e}^{-i t \partial_{x}^{2}}(G-z I d)^{-1} \right)^{-1}\mathrm{e}^{i t \partial_{x}^{2}}T_{u_{0}}T_{\bar{u}_0}\mathrm{e}^{-i t \partial_{x}^{2}}  \left(\frac{\mathrm{e}^{i  t \partial_{x}^{2}} u_{0}(x)- \mathrm{e}^{i  t \partial_{x}^{2}} u_{0}(z)}{x-z}\right)\right] (z) \\ & = \mathrm{e}^{i  t \partial_{x}^{2}} u_{0} (z) - \frac{t}{ i \pi} I_{+}\left[\left( G +2 t L_{u_{0}}-z \mathrm{Id}\right)^{-1} \left(T_{u_{0}}T_{\bar{u}_0}\mathrm{e}^{-i t \partial_{x}^{2}}  \left(\frac{\mathrm{e}^{i  t \partial_{x}^{2}} u_{0}(x)- \mathrm{e}^{i  t \partial_{x}^{2}} u_{0}(z)}{x-z}\right)\right)\right].
\end{align*}
The proof is complete.
\end{proof}
\begin{remark}
In fact, the formula (\ref{4.1}) is also valid for the $H_{+}^s(\mathbb{R})$($s > \frac{1}{2}$) solution to (\ref{1.1}).
\end{remark}
\begin{remark}
From Theorem \ref{theorem 4.1}, if we write $u(t,x)= \mathrm{e}^{i  t \partial_{x}^{2}} u_{0}(x)+ v(t,x)$ as in the proof of Proposition \ref{theorem 3.3}, we can also establish an explicit formula for $v(t,x) \in C\left(\mathbb{R}; H_{+}^2(\mathbb{R})\right)$: For every $t \in \mathbb{R}$ and for every $z \in \mathbb{C}_{+}$, $v(t, z)$ identifies to
\begin{align*}
& \quad -\frac{t}{ i \pi} I_{+}\left[\left( G +2 t L_{u_{0}}-z \mathrm{Id}\right)^{-1} \left(T_{u_{0}}T_{\bar{u}_0}\mathrm{e}^{-i t \partial_{x}^{2}}  \left(\frac{\mathrm{e}^{i  t \partial_{x}^{2}} u_{0}(x)- \mathrm{e}^{i  t \partial_{x}^{2}} u_{0}(z)}{x-z}\right)\right)\right] \\ & = -2t \left[\left(Id + 2t\mathrm{e}^{i t \partial_{x}^{2}}T_{u_{0}}T_{\bar{u}_0}\mathrm{e}^{-i t \partial_{x}^{2}}(G-z I d)^{-1} \right)^{-1}\mathrm{e}^{i t \partial_{x}^{2}}T_{u_{0}}T_{\bar{u}_0}\mathrm{e}^{-i t \partial_{x}^{2}}  \left(\frac{\mathrm{e}^{i  t \partial_{x}^{2}} u_{0}(x)- \mathrm{e}^{i  t \partial_{x}^{2}} u_{0}(z)}{x-z}\right)\right] (z).
\end{align*}
\end{remark}
\noindent As studied in \cite{24, 27, 13, 25}, a natural idea is to use the explicit formula to study the zero dispersion limit of solutions to \eqref{1.1} in  $E_{+}$. However, we lack a suitable entry point to study this problem. In fact, to study the zero dispersion limit,  we shall consider the following equation
\begin{align*}
i \partial_{t} u^{\varepsilon}+\varepsilon\partial_{x}^2 u^{\varepsilon}-2\Pi D\left(|u^{\varepsilon}|^{2}\right)u^{\varepsilon}=0, \quad (t,x ) \in \mathbb{R} \times \mathbb{R}
\end{align*}
with $\varepsilon > 0$. For $u^{\varepsilon}(0,\cdot) = u_0 (\cdot) \in L_{+}^2(\mathbb{R})$, by the $L^2$ conservation, we can deduce that $\|u^{\varepsilon}\|_{L^2} = \|u_0\|_{L^2}$, so there exists a subsequence $\varepsilon_j>0$ tending to $0$ such that $u^{\varepsilon_j}(t)$ has a weak limit in $L_{+}^2(\mathbb{R})$. If all these weak limits coincide, we call this weak limit the zero dispersion limit. Badreddine proved the existence of the zero dispersion limit of solutions to \eqref{1.1} with $u_0 \in L_{+}^2(\mathbb{R})\cap L^{\infty}(\mathbb{R})$ using the explicit formula of \eqref{1.1} \cite{13}. However, for $u^{\varepsilon}(0,\cdot) = u_0(\cdot) \in E$, we cannot find a suitable norm in which the solution  $u^{\varepsilon}(t, \cdot)$  is uniformly bounded with respect to  $\varepsilon>0$, so we lack a suitable entry point to study the zero dispersion limit. Nevertheless, we have discovered an interesting result: using the explicit formula provided in Theorem \ref{theorem 4.1}, we can derive the explicit formula for $u^{\varepsilon}(t,z)$: for all $z \in \mathbb{C}_{+}: = \{z \in \mathbb{C}: \operatorname{Im}(z)>0\}$, we have
\begin{align*}
u^{\varepsilon}(t,z) & =  \mathrm{e}^{i  \varepsilon t \partial_{x}^{2}} u_{0} (z)\\ & -  \frac{t}{ i \pi} I_{+}\left[\left( G +2\varepsilon t D +2 t T_{u_{0}}T_{\bar{u}_{0}}-z \mathrm{Id}\right)^{-1} \left(T_{u_{0}}T_{\bar{u}_0}\mathrm{e}^{-i\varepsilon t \partial_{x}^{2}}  \left(\frac{\mathrm{e}^{i \varepsilon t \partial_{x}^{2}} u_{0}(x)- \mathrm{e}^{i \varepsilon t \partial_{x}^{2}} u_{0}(z)}{x-z}\right)\right)\right].
\end{align*}
From the above formula and Lemma \ref{lemma 2.3}, we can deduce that $u^{\varepsilon}(t,z)$ converges pointwisely in $\mathbb{C}_{+}$ to
\begin{align*}
u_{0} (z) -  \frac{t}{ i \pi} I_{+}\left[\left( G +2 t T_{u_{0}}T_{\bar{u}_{0}}-z \mathrm{Id}\right)^{-1} \left(T_{u_{0}}T_{\bar{u}_0}  \left(\frac{ u_{0}(x)-  u_{0}(z)}{x-z}\right)\right)\right].
\end{align*}
Based on the difficulty in finding a suitable entry point, it is currently difficult to explain how the above pointwise limit is related to any weak limit of $u^{\varepsilon}$.\\\\
Another interesting potential application is using the explicit formula to study the long time behavior of solutions to \eqref{1.1} in  $E_{+}$. Unlike $H_{+}^2(\mathbb{R})$ solutions, there exist solitons to \eqref{1.1} in  $E_{+}$ (see \cite{14} for details), this may lead to the long time behavior of solutions in  $E_{+}$  being more than just scattering.

\end{document}